\documentclass[12pt,reqno]{amsart}

\usepackage[a4paper,margin=2cm]{geometry}
\usepackage[T1]{fontenc}
\usepackage{amsmath,amssymb,amsthm,mathtools}
\usepackage{enumitem}
\usepackage[hidelinks]{hyperref}

\numberwithin{equation}{section}
\allowdisplaybreaks[2]
\setlist{topsep=3pt,itemsep=1pt,parsep=0pt,partopsep=0pt}
\newcommand{\Sym}{\operatorname{Sym}}

\newtheorem{prop}{Proposition}[section]

\newtheorem{cor}[prop]{Corollary}
\newtheorem{lem}[prop]{Lemma}
\newtheorem{ex}[prop]{Example}
\newtheorem{question}[prop]{Question}
\theoremstyle{definition}
\newtheorem{defn}[prop]{Definition}
\theoremstyle{remark}
\newtheorem{rem}[prop]{Remark}

\begin{document}

\title[Mixed Buchsbaum--Rim Multiplicities and Counterexamples]
{Counterexamples to Minkowski-Type Inequalities and the Rees Converse for Modules via Mixed Buchsbaum--Rim Multiplicities}
\author{V. H. Jorge P\'erez}
\address{Instituto de Ci\^encias Matem\'aticas e de Computa\c{c}\~ao, Universidade de S\~ao Paulo, S\~ao Carlos, SP, Brazil}
\email{vhjperez@icmc.usp.br}
\thanks{This work was partially supported by CNPq, grant 304851/2025-6, and FAPESP, grant 2025/21618-0.}

\begin{abstract}
In this paper, we explore mixed Buchsbaum--Rim multiplicities of finite-colength submodules of free modules over Noetherian local rings. We extend the Hilbert-polynomial construction to submodules in arbitrary homogeneous degrees and prove a product formula for the resulting coefficients. We then examine module analogues of Teissier's coefficient inequalities and the Minkowski inequality. Explicit examples show that the coefficientwise power inequalities, log-convexity, and the product-root inequality can all fail; several failures occur over regular local rings. We also prove that equality of integral closures forces equality of all mixed Buchsbaum--Rim multiplicities, while a one-dimensional example shows that the converse fails for arbitrary modules. A converse is recovered when the module closures are the integral closures of modules of the form $IF$ for $\mathfrak m$-primary ideals $I$. The general converse problem in dimension at least two is stated at the end.
\end{abstract}

\maketitle

\section{Introduction}

\paragraph{Standing notation.}
Throughout, unless stated otherwise, $(R,\mathfrak m)$ is a
Noetherian local ring of dimension $d\geq1$, $F$ is a free
$R$-module of rank $p\geq1$, and
\[
S=\operatorname{Sym}_R(F)=\bigoplus_{n\geq0}S_n,
\qquad
F^n:=S_n,
\qquad
q:=d+p-1.
\]
We also fix a finitely generated $R$-module $N$ with
$\dim_R N=d$. Thus $S\otimes_RN$ is a finitely generated graded
$S$-module of dimension $d+p$. Powers and products of submodules are
formed inside $S$. If such a product acts on an $R$-module, the action
means the image of the corresponding multiplication map; we do not
assume that tensoring preserves injectivity.

We begin with the classical ideal-theoretic setting. Let $I,J$ be
$\mathfrak m$-primary ideals of $R$. For $r,s\gg0$, the function
$\ell_R(R/I^rJ^s)$ is a polynomial of total degree $d$, whose
homogeneous component of degree $d$ is
\[
\sum_{i=0}^{d}
\frac{e(I^{[d-i]},J^{[i]})}{(d-i)!\,i!}\,r^{d-i}s^i.
\]
In particular,
\[
e(I^rJ^s)
=
\sum_{i=0}^{d}\binom di\,
e(I^{[d-i]},J^{[i]})\,r^{d-i}s^i.
\]
Teissier's mixed-multiplicity inequalities state that
\begin{equation}\label{CM1}
e(I^{[d-i]},J^{[i]})^d
\leq e(I)^{d-i}e(J)^i
\qquad(0\leq i\leq d),
\end{equation}
and that the mixed multiplicities form a log-convex sequence:
\begin{equation}\label{CM12}
e(I^{[d-i]},J^{[i]})^2
\leq
e(I^{[d-i+1]},J^{[i-1]})\,
e(I^{[d-i-1]},J^{[i+1]})
\qquad(1\leq i\leq d-1).
\end{equation}
These inequalities yield Minkowski's inequality
\begin{equation}\label{eq:ideal-minkowski}
e(IJ)^{1/d}\leq e(I)^{1/d}+e(J)^{1/d}.
\end{equation}

The equality case provides the main motivation for this work. Suppose
that $(R,\mathfrak m)$ is quasi-unmixed of dimension $d\geq2$, and set
\[
e_i=e(I^{[d-i]},J^{[i]}),\qquad 0\leq i\leq d.
\]
The Teissier--Rees--Sharp--Katz criterion relates equality in
\eqref{eq:ideal-minkowski} to equality of the integral closures of
powers of $I$ and $J$. More precisely, for positive integers $r,s$,
\[
\overline{I^r}=\overline{J^s}
\quad\Longleftrightarrow\quad
\frac{e_i}{e_{i-1}}=\frac{r}{s}
\quad\text{for every }1\leq i\leq d,
\]
under the stated hypotheses; see
\cite{Teissier0,Teissier1,RS,KatzMinkowski,Verma}.
Consequently, if $e_0=e_1=\cdots=e_d=c$, then the ratios
$e_i/e_{i-1}$ are all equal to $1$, and the criterion gives
$\overline{I^r}=\overline{J^r}$ for some $r>0$. Hence
$\overline I=\overline J$.

We now turn to modules. Let $E_1,E_2\subseteq F$ be submodules of
finite colength. For $r,s\geq0$, set
\[
h_{E_1,E_2;N}(r,s)
=
\ell_R\!\left(
(S_{r+s}/E_1^rE_2^s)\otimes_RN
\right).
\]
For $r,s\gg0$, this function agrees with a polynomial of total degree
at most $q$. Its homogeneous component of degree $q$ has the form
\[
\sum_{i=0}^{q}
\frac{
e_{\mathrm{BR}}(E_1^{[q-i]},E_2^{[i]};N)
}{(q-i)!\,i!}\,r^{q-i}s^i,
\]
where the coefficients are the mixed Buchsbaum--Rim multiplicities.
When $N=R$, we omit $N$ from the notation. More generally, we also
consider submodules in arbitrary homogeneous degrees. The theory of
mixed Buchsbaum--Rim multiplicities was developed by Kirby--Rees and
Kleiman--Thorup; see
\cite{Buchsbaum-Rim,Kirby-Rees1,Kleiman-Thorup,Kleiman-Thorup2}.
Further results appear in \cite{Bedregal-Perez}.

This setting suggests several questions. Do the ideal-theoretic
inequalities \eqref{CM1} and \eqref{CM12} extend to submodules of finite
colength? The corresponding power-type inequality would be
\begin{equation}\label{CM3}
e_{\mathrm{BR}}(E_1^{[q-i]},E_2^{[i]})^q
\leq
e_{\mathrm{BR}}(E_1)^{q-i}
e_{\mathrm{BR}}(E_2)^i
\qquad(1\leq i\leq q-1),
\end{equation}
while the log-convexity analogue would be
\begin{equation}\label{CM4}
\begin{aligned}
e_{\mathrm{BR}}(E_1^{[q-i]},E_2^{[i]})^2
\leq{}&
e_{\mathrm{BR}}(E_1^{[q-i+1]},E_2^{[i-1]})\cdot e_{\mathrm{BR}}(E_1^{[q-i-1]},E_2^{[i+1]})
\,\,\,(1\leq i\leq q-1).
\end{aligned}
\end{equation}
One may also ask whether the multiplicity associated with the product
$E_1E_2\subseteq S_2$ satisfies a Minkowski-type inequality. Writing
$\widetilde e_{\mathrm{BR}}(E_1E_2)$ for the multiplicity defined by
the pure length function
$\ell_R(S_{2n}/(E_1E_2)^n)$, this would take the form
\begin{equation}\label{CM5}
\widetilde e_{\mathrm{BR}}(E_1E_2)^{1/q}
\leq
e_{\mathrm{BR}}(E_1)^{1/q}
+
e_{\mathrm{BR}}(E_2)^{1/q}.
\end{equation}

A second question concerns integral closure. Equality
$\overline{E_1}=\overline{E_2}$ implies equality of all mixed
Buchsbaum--Rim multiplicities:
\[
e_{\mathrm{BR}}(E_1^{[q-i]},E_2^{[i]})
=
e_{\mathrm{BR}}(E_1)
=
e_{\mathrm{BR}}(E_2)
\qquad(0\leq i\leq q).
\]
The natural question is the following.

\begin{question}\label{converse}
Assume that $R$ is quasi-unmixed with infinite residue field, and let
$E_1,E_2\subseteq F$ have finite colength. Does
\[
e_{\mathrm{BR}}(E_1^{[q-i]},E_2^{[i]})
=
e_{\mathrm{BR}}(E_1)
=
e_{\mathrm{BR}}(E_2)
\qquad(0\leq i\leq q)
\]
imply $\overline{E_1}=\overline{E_2}$ in $F$?
\end{question}

The answer is negative without further hypotheses. We give an explicit
counterexample over a one-dimensional complete, reduced,
Cohen--Macaulay local ring; in that example, all the mixed
Buchsbaum--Rim multiplicities agree, but the integral closures differ.
We also prove a converse in a restricted class in which the module
closures arise from $\mathfrak m$-primary ideals.

The question above is distinct from the module converse to Rees's
mixed-multiplicity theorem proved by Ferrari, Jorge-Pérez, and Merighe
\cite{FerrariJorgePerezMerighe}. Their result gives criteria for a
sequence to be a joint reduction from an equality between a
Buchsbaum--Rim multiplicity and a mixed Buchsbaum--Rim multiplicity,
under explicit height and radical hypotheses. Here the issue is
different: we compare the entire mixed Buchsbaum--Rim coefficient
vector with the two endpoint multiplicities and ask whether these
numerical equalities determine the integral closure.

The main results of this paper are as follows.
\begin{enumerate}
\item We prove eventual polynomiality for mixed length functions
associated with submodules in arbitrary homogeneous degrees and
establish the corresponding product formula.
\item We construct explicit counterexamples to
\eqref{CM3}, \eqref{CM4}, and \eqref{CM5}.
\item We prove an aggregate upper bound in terms of the annihilator
ideals of $F/E_1$ and $F/E_2$.
\item We prove that equality of integral closures implies equality of
all mixed Buchsbaum--Rim multiplicities. We also give a
counterexample to the converse and prove a positive converse for
modules whose integral closures arise from ideals.
\end{enumerate}

The paper is organized as follows. Section~\ref{sec:preliminaries}
recalls Rees algebras, integral closure, and reductions for modules.
Section~\ref{sec:BR-mixed} develops the relevant Hilbert polynomials,
and Section~\ref{sec:product} proves the product formula.
Section~\ref{sec:minkowski} studies the module analogues of the
coefficient inequalities and Minkowski's inequality.
Finally, Section~\ref{sec:integral-closure} addresses
Question~\ref{converse}.

\section{Preliminaries on Rees algebras and integral closure}
\label{sec:preliminaries}

For a submodule $E\subseteq F$, its Rees algebra is the graded
$R$-subalgebra
\[
\mathcal R(E)
=
\operatorname{im}\bigl(\operatorname{Sym}_R(E)\longrightarrow S\bigr)
=
\bigoplus_{n\geq0}E^n,
\]
where $E^n=[\mathcal R(E)]_n\subseteq S_n$.

An inclusion $U\subseteq E$ is a reduction if
\[
E^{n+1}=UE^n
\qquad\text{for all }n\gg0.
\]
The integral closure of $E$ in $F$ is
\[
\overline E
=
\left(\overline{\mathcal R(E)}^{\,S}\right)_1,
\]
where the bar on the right denotes integral closure in $S$. Equivalently,
$x\in F$ belongs to $\overline E$ if and only if it satisfies an
equation
\[
x^n+c_1x^{n-1}+\cdots+c_n=0
\quad\text{in }S,
\qquad c_i\in E^i.
\]
More generally, we write
\[
\overline{E^n}
=
\left(\overline{\mathcal R(E)}^{\,S}\right)_n.
\]

For $U\subseteq E\subseteq F$, the following conditions are
equivalent:
\begin{enumerate}
\item $U$ is a reduction of $E$;
\item $\mathcal R(E)$ is integral over $\mathcal R(U)$;
\item $E\subseteq\overline U$.
\end{enumerate}
Equivalently, if $I_E=ES$ and $I_U=US$ are the ideals of $S$
generated by $E$ and $U$, then $U$ is a reduction of $E$ if and only
if $I_U$ is a reduction of $I_E$; see
\cite[Theorem~16.2.3]{SH}.

\section{Preliminaries on Rees algebras and integral closure}
\label{sec:preliminaries}

For a submodule $E\subseteq F$, its Rees algebra is the image
\[
\mathcal R(E)
=
\operatorname{im}\bigl(\Sym_R(E)\longrightarrow \Sym_R(F)\bigr)
=
\bigoplus_{n\geq0}E^n,
\]
where $E^n=[\mathcal R(E)]_n\subseteq S_n$.

An inclusion $U\subseteq E$ is a reduction if
\[
E^{n+1}=UE^n
\qquad\text{for all }n\gg0.
\]
The integral closure of $E$ in $F$ is
\[
\overline E
=
\left(\overline{\mathcal R(E)}^{\,S}\right)_1.
\]
Equivalently, $x\in F$ belongs to $\overline E$ if and only if it satisfies an equation
\[
x^n+c_1x^{n-1}+\cdots+c_n=0
\quad\text{in }S,
\qquad c_i\in E^i.
\]
More generally, we write
\[
\overline{E^n}
=
\left(\overline{\mathcal R(E)}^{\,S}\right)_n.
\]

For $U\subseteq E\subseteq F$, the following are equivalent:
\begin{enumerate}
\item $U$ is a reduction of $E$;
\item $\mathcal R(E)$ is integral over $\mathcal R(U)$;
\item $E\subseteq\overline U$.
\end{enumerate}
Equivalently, if $I_E=ES$ and $I_U=US$ are the ideals of $S$ generated by $E$ and $U$, then $U$ is a reduction of $E$ if and only if $I_U$ is a reduction of $I_E$. See \cite[Theorem~16.2.3]{SH}.

\section{Hilbert polynomials and mixed Buchsbaum--Rim multiplicities}\label{sec:BR-mixed}

\subsection{Buchsbaum--Rim multiplicity}

Let $E\subseteq F$ satisfy $0<\ell_R(F/E)<\infty$.
By Buchsbaum and Rim \cite[Section~3]{Buchsbaum-Rim}, the function
$\ell_R((F^n/E^n)\otimes_RN)$ agrees, for $n\gg0$, with a polynomial
of degree at most $q$. We write its expansion as
\[
\ell_R((F^n/E^n)\otimes_RN)
=\frac{e_{\mathrm{BR}}(E;N)}{q!}n^q
+\text{terms of degree less than }q.
\]
The coefficient $e_{\mathrm{BR}}(E;N)$ is the Buchsbaum--Rim
multiplicity of $E$ with coefficients in $N$; see also \cite{SH}.
Equivalently,
\[
e_{\mathrm{BR}}(E;N)=
q!\lim_{n\to \infty}\frac{\ell_R((F^n/E^n)\otimes_RN)}{n^q}.
\]
When $N=R$, we write $e_{\mathrm{BR}}(E)=e_{\mathrm{BR}}(E;R)$.
The same polynomial convention gives $e_{\mathrm{BR}}(F;N)=0$.

Over a formally equidimensional local ring, equality of
Buchsbaum--Rim multiplicities characterizes reductions of nested
submodules of finite colength. This generalizes Rees' theorem for
ideals; see \cite{Buchsbaum-Rim,Katz}, \cite[p.~317]{SH}, and
\cite[Theorem~5.7]{Kleiman-Thorup}.

\begin{lem}[{\cite[Corollary~16.5.7]{SH}}]\label{Coro16.5.7SH}
Suppose that $R$ is formally equidimensional, and let
$E_1\subseteq E_2\subseteq F$ satisfy $\ell_R(F/E_1)<\infty$.
Then $E_1$ is a reduction of $E_2$ in $F$ if and only if
$e_{\mathrm{BR}}(E_1)=e_{\mathrm{BR}}(E_2)$.
\end{lem}

\subsection{Multiplicity in arbitrary homogeneous degree} \label{sec:br-arbitrary-degree}

Fix an integer $e\geq1$ and a submodule $E\subseteq S_e$
such that
\begin{equation}\label{eq:br-degree-colength}
\ell_R\bigl((S_e/E)\otimes_RN\bigr)<\infty.
\end{equation}
In particular, $E^n\subseteq S_{en}$.

The classical multiplicity was introduced by Buchsbaum and Rim
\cite{Buchsbaum-Rim}. For graded developments, see \cite{Kirby-Rees1}; for geometric
approaches, see \cite{Kleiman-Thorup,Kleiman-Thorup2}. See also \cite{SH}.
We first establish the required univariate polynomiality and then
identify its leading coefficient with the multiplicity of
\cite[Section~8]{Kleiman-Thorup2}.

\begin{prop}\label{prop:br-degree-polynomiality}
Under the preceding assumptions, there exists a polynomial
$P_E(n;N)\in\mathbb Q[n]$ of degree at most $q$ such that
\[
P_E(n;N)=\ell_R\bigl((S_{en}/E^n)\otimes_RN\bigr)
\qquad(n\gg0).
\]
\end{prop}

\begin{proof}
Give the Veronese algebra $B=S^{(e)}$ its standard grading
$B_n=S_{en}$, and set $T=B\otimes_RN$ and $J=EB$.
Then $T$ is generated in degree zero, and right exactness gives
\[
H(n):=\ell_R\bigl((S_{en}/E^n)\otimes_RN\bigr)
\]
Equip $G=\operatorname{gr}_J(T)$ with the bigrading
\[
G_{i,j}=\frac{E^iT_j}{E^{i+1}T_{j-1}}
\qquad(i,j\geq0),
\]
where $T_{-1}=0$. The algebra $\operatorname{gr}_J(B)$ is generated
over $R$ in bidegrees $(1,0)$ and $(0,1)$, and $G$ is finite over it.
Consequently, $K=\bigoplus_{i\geq0,\,j\geq1}G_{i,j}$ is a finite
bigraded submodule.

By \eqref{eq:br-degree-colength}, there is $c\geq1$ with
$\mathfrak m^cT_1\subseteq ET_0$. Since $T_j=B_{j-1}T_1$, this implies
$\mathfrak m^cE^iT_j\subseteq E^{i+1}T_{j-1}$ for $j\geq1$.
Thus $\mathfrak m^cK=0$. Using total degree $i+j$, Hilbert--Serre
shows that $\ell_R(K_n)$ is eventually polynomial. The filtration
\[
T_n\supseteq ET_{n-1}\supseteq\cdots\supseteq E^nT_0
\]
gives $H(n)=\sum_{i=0}^{n-1}\ell_R(G_{i,n-i})=\ell_R(K_n)$,
proving polynomiality.

Iterating $\mathfrak m^cT_1\subseteq ET_0$ yields
$\mathfrak m^{cn}T_n\subseteq E^nT_0$. Therefore
\[
H(n)\leq\ell_R(T_n/\mathfrak m^{cn}T_n)
 =\binom{en+p-1}{p-1}\ell_R(N/\mathfrak m^{cn}N).
\]
Since $\dim_RN=d$, one has
$\ell_R(N/\mathfrak m^{cn}N)=O(n^d)$, while
$\binom{en+p-1}{p-1}=O(n^{p-1})$. Hence $H(n)=O(n^q)$,
which proves the degree bound.
which proves the degree bound.

\end{proof}

\begin{defn}\label{def:br-degree}
The Buchsbaum--Rim multiplicity of $E$ in degree $e$, with
coefficients in $N$ and relative to $S$, is defined by
\[
\widetilde e_{\mathrm{BR}}(E;N)
:=q!\lim_{n\to\infty}
\frac{\ell_R((S_{en}/E^n)\otimes_RN)}{n^q}.
\]
When $N=R$, we omit $N$
from the notation.
\end{defn}

\begin{rem}\label{rem:br-degree-conventions}
The number $\widetilde e_{\mathrm{BR}}(E;N)$ is a nonnegative integer:
$P_E(n;N)$ is eventually integer-valued and nonnegative, so its
normalized coefficient in degree $q$ is integral and nonnegative.
It may vanish; for example, it vanishes when $E=S_e$.
For $e=1$, the definition is the classical Buchsbaum--Rim
multiplicity $e_{\mathrm{BR}}(E;N)$. For $e>1$, it is relative to
multiplication in $S$, and must not be confused with the classical
multiplicity of the inclusion into the free module $S_e$, whose
ambient symmetric algebra would instead be
$\operatorname{Sym}_R(S_e)$.
\end{rem}

\begin{lem}\label{lem:br-degree-fiber-bound}
Let $(A,\mathfrak n)$ be a Noetherian local ring of dimension
$\delta\geq1$, let $C=A[t_1,\ldots,t_p]$, and let $U\subseteq C_e$
have finite colength, where $e\geq1$. Grade $\mathcal A=A[U]$ by
assigning degree one to $U$. Then
\[
\dim(\mathcal A/\mathfrak n\mathcal A)\leq\delta+p-1.
\]
\end{lem}

\begin{proof}
It suffices to consider the images of $\mathcal A$ over the minimal
prime quotients of $A$. Indeed, the kernel of the map to the product
of these images is nilpotent, and the resulting finitely many
closed subsets cover $\operatorname{Spec}(\mathcal A)$.

Suppose first that $A$ is a domain of positive dimension
$\delta'$. On $X=\mathbb P_A^{p-1}$, the forms in $U$ define a
linear system which is the complete degree-$e$ system outside the
closed fiber. Its graph closure $Y$ is the blowup of its base ideal.
It is integral, has dimension $\delta'+p-1$, and maps properly and
surjectively to $\operatorname{Proj}(A[U])$. Surjectivity follows
because its generic image is dense and the morphism is proper.
The closed fiber $Y_{\mathfrak n}$ is a proper closed subset of $Y$,
so its dimension is at most $\delta'+p-2$. Surjectivity on closed
fibers now gives
$\dim(A[U]/\mathfrak nA[U])\leq\delta'+p-1$.
If the domain quotient has dimension zero, it is a field, and its
algebra of forms has dimension at most $p$. Both bounds are at most
$\delta+p-1$.
\end{proof}

\begin{lem}\label{lem:br-degree-shift}
With $B=S^{(e)}$, $T=B\otimes_RN$, and $H(n)=\ell_R(T_n/E^nT_0)$,
for every fixed integer $h\geq1$ one has
\[
\ell_R(T_{n+h}/E^nT_h)=H(n+h)+O(n^{q-1}).
\]
\end{lem}

\begin{proof}
Set $R'=R/\operatorname{Ann}_R(N)$ and let $E'$ be the image of $E$
in $S_e\otimes_RR'$. The module
$(S_e\otimes_RR')/E'$ has finite length: outside the closed point
its tensor product with the nonzero localized module $N$ vanishes,
so Nakayama's lemma forces it to vanish. Also $\dim R'=d$.
Thus Lemma~\ref{lem:br-degree-fiber-bound}, applied to
$\mathcal A=R'[E']$, gives
$\dim(\mathcal A/\mathfrak m\mathcal A)\leq q$.

For fixed $h$, the graded module
\[
V_h=\bigoplus_{n\geq0}\frac{E^nT_h}{E^{n+h}T_0}
\]
is finite over $\mathcal A$, since its numerator is generated by
$T_h$ in degree zero. Choose $c$ as in the proof of
Proposition~\ref{prop:br-degree-polynomiality}. Then
$\mathfrak m^{ch}T_h\subseteq E^hT_0$, whence
$\mathfrak m^{ch}V_h=0$ and $\dim_{\mathcal A}V_h\leq q$.
Hilbert--Serre yields
$\ell_R((V_h)_n)=O(n^{q-1})$.
The assertion follows from the exact sequence
\[
0\longrightarrow\frac{E^nT_h}{E^{n+h}T_0}
\longrightarrow\frac{T_{n+h}}{E^{n+h}T_0}
\longrightarrow\frac{T_{n+h}}{E^nT_h}
\longrightarrow0.
\]
\end{proof}

\begin{prop}\label{prop:br-degree-KT-identification}
Under the assumptions of Proposition~\ref{prop:br-degree-polynomiality},
one has
\[
\widetilde e_{\mathrm{BR}}(E;N)=e^{q,0}(E;N),
\]
Here $e^{q,0}(E;N)$ is the Buchsbaum--Rim multiplicity of
Kleiman--Thorup \cite[Section~8]{Kleiman-Thorup2}, computed for
$S\otimes_RN$ and the degree-$e$ submodule $E$.
\end{prop}

\begin{proof}
Condition~\eqref{eq:br-degree-colength} implies that the base locus
of $E$ on $\operatorname{Proj}(S)$ meets the support of
$\widetilde{S\otimes_RN}$ only in the closed fiber. Hence
\cite[Proposition~8.2]{Kleiman-Thorup2} applies to
\[
\lambda(n,b)=\ell_R\left(
\frac{S_{en+b}\otimes_RN}{E^n(S_b\otimes_RN)}\right).
\]
For $n,b\gg0$, this is a polynomial of total degree at most $q$,
whose homogeneous component of degree $q$ is
\[
\sum_{i+k=q}\frac{e^{i,k}(E;N)}{i!\,k!}n^ib^k.
\]
Choose a fixed $h\geq1$ sufficiently large that $b=eh$ belongs to
the polynomial range. Then
\[
\lambda(n,eh)=\frac{e^{q,0}(E;N)}{q!}n^q+O(n^{q-1}).
\]
On the other hand, Lemma~\ref{lem:br-degree-shift} gives
\[
\lambda(n,eh)=H(n+h)+O(n^{q-1})
\]
Comparison proves the assertion.
\end{proof}

\begin{rem}\label{rem:br-degree-specialization}
The preceding argument does not specialize the polynomial identity
of \cite[Proposition~8.2]{Kleiman-Thorup2} at $b=0$, which need not
lie in its stated range of validity. Instead, it fixes a sufficiently
large value $b=eh$ and uses Lemma~\ref{lem:br-degree-shift} to compare
the leading coefficients.
\end{rem}

\begin{rem}\label{remark1}
Let $M\subseteq F$ have finite colength, and let $r\geq1$.
Regard $M^r$ as a submodule of $S_r$. Since
$(M^r)^n=M^{rn}$ inside $S_{rn}$, the defining Hilbert functions
satisfy
\[
\ell_R\!\left((S_{rn}/(M^r)^n)\otimes_RN\right)
=\ell_R\!\left((S_{rn}/M^{rn})\otimes_RN\right).
\]
The leading term on the right is
$\widetilde e_{\mathrm{BR}}(M;N)(rn)^q/q!$.
Consequently,
\[
\widetilde e_{\mathrm{BR}}(M^r;N)
=r^q\widetilde e_{\mathrm{BR}}(M;N).
\]
In particular, for $N=R$,
$\widetilde e_{\mathrm{BR}}(M^r)=r^{d+p-1}e_{\mathrm{BR}}(M)$.
\end{rem}

\subsection{Mixed Buchsbaum--Rim multiplicities}

Let $E_1,E_2\subseteq F$ satisfy $\ell_R(F/E_i)<\infty$ for
$i=1,2$. The function
\[
h_R(n_1,n_2;N)=
\ell_R\!\left(
(S_{n_1+n_2}/E_1^{n_1}E_2^{n_2})\otimes_RN
\right)
\]
agrees, for $n_1,n_2\gg0$, with a polynomial in
$\mathbb Q[n_1,n_2]$ of total degree at most $q$; see the remark
below and Proposition~\ref{prop:mixed-br-arbitrary-degree}.
Its homogeneous component of total degree $q$ is written as
\[
\sum_{\substack{a,b\geq0\\a+b=q}}
\frac{e_{\mathrm{BR}}(E_1^{[a]},E_2^{[b]};N)}
{a!\,b!}\,n_1^a n_2^b.
\]
The coefficients
$e_{\mathrm{BR}}(E_1^{[a]},E_2^{[b]};N)$ are called the
mixed Buchsbaum--Rim multiplicities of $E_1$ and $E_2$
of type $(a,b)$ with coefficients in $N$.
The bracketed exponents specify the multiplicity type
and do not denote module powers.

For $N=R$, we write
\[
e_{\mathrm{BR}}(E_1^{[a]},E_2^{[b]})
:=
e_{\mathrm{BR}}(E_1^{[a]},E_2^{[b]};R).
\]
When $p=1$ and $E_1,E_2$ are proper, they are
$\mathfrak m$-primary ideals, and this construction recovers their
mixed Hilbert--Samuel multiplicities with coefficients in $N$.

\begin{rem}
Combining \cite[Proposition~8.2]{Kleiman-Thorup2} with
multigraded Hilbert--Serre yields a polynomial in
$\mathbb Q[n_1,n_2]$ that agrees with $h_R(n_1,n_2;N)$
for $n_1,n_2\gg0$.
Indeed, one applies these results to the two outer terms of
the exact sequence
\[
0\longrightarrow
\frac{E_2^{n_2}T_{n_1}}{E_1^{n_1}E_2^{n_2}T_0}
\longrightarrow
\frac{T_{n_1+n_2}}{E_1^{n_1}E_2^{n_2}T_0}
\longrightarrow
\frac{T_{n_1+n_2}}{E_2^{n_2}T_{n_1}}
\longrightarrow0,
\]
where $T=S\otimes_RN$, using the associated graded filtration
induced by $E_1$ for the left-hand term.
Moreover, choosing $c\geq1$ with
$\mathfrak m^cF\subseteq E_1\cap E_2$ and setting $u=n_1+n_2$,
we obtain
\[
0 \leq h_R(n_1,n_2;N)
\leq
\binom{u+p-1}{p-1}\ell_R(N/\mathfrak m^{cu}N)
\]
Since eventual polynomiality has already been established, this
uniform estimate bounds the total degree by $d+p-1$.
\end{rem}

\subsection{Mixed multiplicities in arbitrary homogeneous degrees}

Fix positive integers $e_1,e_2$ and submodules
$E_i\subseteq S_{e_i}=F^{e_i}$ satisfying
$\ell_R(S_{e_i}/E_i)<\infty$ for $i=1,2$.
All products are taken in $S$.

\begin{prop}\label{prop:mixed-br-arbitrary-degree}
Fix positive integers $e_1,e_2$, and let
$E_i\subseteq S_{e_i}=F^{e_i}$ be submodules satisfying
$\ell_R(S_{e_i}/E_i)<\infty$ for $i=1,2$. Then the function
\[
h_R(n_1,n_2;N) = \ell_R\left( \left( S_{e_1n_1+e_2n_2}/E_1^{n_1}E_2^{n_2} \right)\otimes_RN \right)
\]
agrees, for $n_1,n_2\gg0$, with a polynomial in $\mathbb Q[n_1,n_2]$ of total degree at most $q$.
\end{prop}

\begin{proof}
Set $T=S\otimes_RN$, $T_v=S_v\otimes_RN$. All products acting on $T$ are understood as the images of the corresponding multiplication maps. By right exactness of tensor products, $\left( S_u/E_1^{n_1}E_2^{n_2} \right)\otimes_RN \cong T_u/E_1^{n_1}E_2^{n_2}T_0$, where $u=e_1n_1+e_2n_2$.

Since $E_1^{n_1}E_2^{n_2}T_0 \subseteq E_2^{n_2}T_{e_1n_1} \subseteq T_u$, we obtain an exact sequence
\[
0\longrightarrow \frac{E_2^{n_2}T_{e_1n_1}}{E_1^{n_1}E_2^{n_2}T_0} \longrightarrow \frac{T_u}{E_1^{n_1}E_2^{n_2}T_0} \longrightarrow \frac{T_u}{E_2^{n_2}T_{e_1n_1}} \longrightarrow 0.
\]
Consequently, $h_R(n_1,n_2;N)=A(n_1,n_2)+B(n_1,n_2)$, where $A(n_1,n_2)$ and $B(n_1,n_2)$ denote the lengths of the left-hand and right-hand terms, respectively.

We first consider $B(n_1,n_2)$. Since $\ell_R(S_{e_2}/E_2)<\infty$, the hypotheses of \cite[Proposition~8.2]{Kleiman-Thorup2} are satisfied. Hence there exists a polynomial $Q(a,b)\in\mathbb Q[a,b]$ such that $\ell_R\left( T_{e_2a+b}/E_2^aT_b \right) = Q(a,b)$ for $a,b\gg0$. Taking $a=n_2$ and $b=e_1n_1$, we obtain $B(n_1,n_2)=Q(n_2,e_1n_1)$ for $n_1,n_2\gg0$. Thus $B$ is eventually polynomial in $n_1,n_2$.

We now prove the eventual polynomiality of $A$. Consider the externally bigraded algebra and module $\mathcal C = \bigoplus_{j,b\geq0}E_2^bS_{e_1j}$ and $\mathcal T = \bigoplus_{j,b\geq0}E_2^bT_{e_1j}$, where the summand $E_2^bS_{e_1j}$ has bidegree $(j,b)$.

Since $S_{e_1j}=(S_{e_1})^j$, the algebra $\mathcal C$ is generated over $R$ by its components of degrees $(1,0)$ and $(0,1)$. Moreover, $\mathcal T$ is a finite graded $\mathcal C$-module generated by $\mathcal T_{0,0}=T_0=N$.

Let $\mathcal J$ be the ideal of $\mathcal C$ generated by $E_1$ in bidegree $(1,0)$. Set $\mathcal A=\operatorname{gr}_{\mathcal J}(\mathcal C)$ and $G=\operatorname{gr}_{\mathcal J}(\mathcal T)$. Then $G$ is a finite graded module over the Noetherian algebra $\mathcal A$.

Combining the associated grading with the original bigrading, and writing $j$ for the residual first degree after removing the $\mathcal J$-filtration degree, we obtain
\[
G_{i,j,b} = \frac{ E_1^iE_2^bT_{e_1j} }{ E_1^{i+1}E_2^bT_{e_1(j-1)} }, \qquad i,j,b\geq0,
\]
where the denominator is understood to be zero when $j=0$.

The generators of $\mathcal A$ have tridegrees $(1,0,0)$, $(0,1,0)$, and $(0,0,1)$. The submodule
\[
K=\bigoplus_{i,b\geq0,\,j\geq1}G_{i,j,b}
\]
is stable under the action of $\mathcal A$, and hence is finite over $\mathcal A$.

Choose $c\geq1$ such that $\mathfrak m^cS_{e_1}\subseteq E_1$. For $j\geq1$, the equality $T_{e_1j}=S_{e_1}T_{e_1(j-1)}$ implies $\mathfrak m^cE_1^iE_2^bT_{e_1j} \subseteq E_1^{i+1}E_2^bT_{e_1(j-1)}$. Thus $\mathfrak m^cK=0$.

Now replace the trigrading by the bigrading $(i,j,b)\longmapsto(i+j,b)$. Under this regrading, $\mathcal A$ becomes a standard bigraded algebra generated in degrees $(1,0)$ and $(0,1)$. Therefore $K$ is a finite bigraded module over $\mathcal A/\mathfrak m^c\mathcal A$, whose degree-zero component is the Artinian ring $R/\mathfrak m^c$.

The multigraded Hilbert--Serre theorem gives a polynomial
$P_K(X,Y)\in\mathbb Q[X,Y]$ such that
\[
\ell_R(K_{n_1,n_2})=P_K(n_1,n_2)
\qquad(n_1,n_2\gg0).
\]

By the definition of the regrading, $K_{n_1,n_2} = \bigoplus_{i=0}^{n_1-1} G_{i,n_1-i,n_2}$. The filtration $$E_2^{n_2}T_{e_1n_1} \supseteq E_1E_2^{n_2}T_{e_1(n_1-1)} \supseteq\cdots\supseteq E_1^{n_1}E_2^{n_2}T_0$$ has successive quotients $G_{i,n_1-i,n_2}$ for $0\leq i\leq n_1-1$. Therefore,
\[
A(n_1,n_2) = \sum_{i=0}^{n_1-1} \ell_R\left(G_{i,n_1-i,n_2}\right) = \ell_R(K_{n_1,n_2}) = P_K(n_1,n_2)
\]
for $n_1,n_2\gg0$. Hence $h_R(n_1,n_2;N) = P_K(n_1,n_2)+Q(n_2,e_1n_1)$ for all sufficiently large $n_1,n_2$. Thus $h_R$ agrees eventually with a polynomial $P(n_1,n_2)\in\mathbb Q[n_1,n_2]$.

It remains to bound the total degree of $P$. Choose $c\geq1$ such that $\mathfrak m^cS_{e_i}\subseteq E_i$ ($i=1,2$). Since $S$ is a polynomial ring, $(S_{e_1})^{n_1}(S_{e_2})^{n_2}=S_u$ for $u=e_1n_1+e_2n_2$. Consequently, $\mathfrak m^{c(n_1+n_2)}T_u \subseteq E_1^{n_1}E_2^{n_2}T_0$. Since $T_u\cong N^{\binom{u+p-1}{p-1}}$ as an $R$-module, we obtain
\[
\begin{aligned}
0\leq h_R(n_1,n_2;N)
&\leq \ell_R\!\left(T_u/\mathfrak m^{c(n_1+n_2)}T_u\right)\\
&=\binom{u+p-1}{p-1}
\ell_R\!\left(N/\mathfrak m^{c(n_1+n_2)}N\right)\\
&=O\!\left((n_1+n_2)^q\right).
\end{aligned}
\]

Suppose that $P$ had total degree $D>q$, and let $P_D$ denote its homogeneous component of degree $D$. Since $P_D\neq0$, there exist positive integers $a,b$ such that $P_D(a,b)\neq0$. Hence $P(an,bn) = P_D(a,b)n^D+O(n^{D-1})$, which contradicts the estimate $P(an,bn)=O(n^q)$. Therefore, $\deg P\leq q=d+p-1$. This proves the proposition.
\end{proof}

The homogeneous component of total degree $q$ of this polynomial
is written as
\[
\sum_{\substack{a,b\geq0\\a+b=q}}
\frac{\widetilde e_{\mathrm{BR}}
(E_1^{[a]},E_2^{[b]};N)}
{a!\,b!}\,n_1^a n_2^b.
\]
These coefficients are called the mixed Buchsbaum--Rim
multiplicities of $E_1,E_2$ of type $(a,b)$ in degrees
$e_1,e_2$, with coefficients in $N$.
When $N=R$, we omit $N$ from the notation.
For $e_1=e_2=1$, this definition recovers
$e_{\mathrm{BR}}(E_1^{[a]},E_2^{[b]};N)$.

\begin{rem}\label{notation}
For $u=e_1n_1+e_2n_2$, right exactness gives the natural
isomorphism
\[
(S_u/E_1^{n_1}E_2^{n_2})\otimes_RN
\cong
\frac{S_u\otimes_RN}
{\operatorname{Im}\!\left(
E_1^{n_1}E_2^{n_2}\otimes_RN
\longrightarrow S_u\otimes_RN
\right)}.
\]
Thus, with $T=S\otimes_RN$, the function above is precisely
$\ell_R(T_u/E_1^{n_1}E_2^{n_2}T_0)$.
No injectivity of the tensor-induced map is assumed.
\end{rem}

\begin{rem}\label{rem:comparison-graded-notation}
The degree-one conventions above agree with the graded formulation
in \cite[Remark~2.1 and Section~4]{Bedregal-Perez}.
Indeed, take $G=S=\operatorname{Sym}_R(F)$ and
\[
M=S\otimes_RN
=\bigoplus_{v\geq0}M_v
=\bigoplus_{v\geq0}(F^v\otimes_RN),
\]
which has dimension $d+p$. If $E\subseteq F$ and
$\mathcal R(E)=\bigoplus_{n\geq0}\mathcal R_n(E)$ is its Rees
algebra, then $\mathcal R_n(E)=E^n$, and
\[
\begin{aligned}
\ell_R\!\left(M_n/\mathcal R_n(E)M_0\right)
&=\ell_R\!\left(
\frac{F^n\otimes_RN}
{\operatorname{Im}(E^n\otimes_RN\longrightarrow F^n\otimes_RN)}
\right)\\
&=\ell_R\!\left((F^n/E^n)\otimes_RN\right).
\end{aligned}
\]
Likewise, for $E_i\subseteq S_{e_i}$ and $u=e_1n_1+e_2n_2$,
the graded-module expression $\ell_R(M_u/E_1^{n_1}E_2^{n_2}M_0)$ equals
\(
\ell_R\!\left(
\left(S_u/E_1^{n_1}E_2^{n_2}\right)\otimes_RN
\right),
\)
as explained in Remark~\ref{notation}.
\end{rem}

\section{A product formula for mixed Buchsbaum--Rim multiplicities}\label{sec:product}

This section proves the diagonal product formula for mixed
Buchsbaum--Rim multiplicities. We first treat submodules of degree
one, then record its version for submodules in arbitrary homogeneous
degrees.

\begin{lem}\label{lem:mixed-product-formula}
Let $E_1,E_2\subseteq F$ have finite colength. Regarding
$E_1E_2$ as a submodule of
$F^2=\operatorname{Sym}_R^2(F)$, one has
\[
\widetilde e_{\mathrm{BR}}(E_1E_2;N)
=
\sum_{i=0}^{q}
\binom{q}{i}
e_{\mathrm{BR}}
\left(E_1^{[q-i]},E_2^{[i]};N\right).
\]
\end{lem}

\begin{proof}
Set $S=\operatorname{Sym}_R(F)$.
First, $E_1E_2$ has finite colength in $S_2$:
if $\mathfrak m^cF\subseteq E_1\cap E_2$, then
$\mathfrak m^{2c}S_2\subseteq E_1E_2$.
Thus $\widetilde e_{\mathrm{BR}}(E_1E_2;N)$ is defined.

By Proposition~\ref{prop:mixed-br-arbitrary-degree}, for
$n_1,n_2\gg0$ one has
\[
\ell_R\!\left(
(S_{n_1+n_2}/E_1^{n_1}E_2^{n_2})\otimes_RN
\right)
=
\sum_{i=0}^{q}
\frac{e_{\mathrm{BR}}(E_1^{[q-i]},E_2^{[i]};N)}
{(q-i)!\,i!}\,
n_1^{q-i}n_2^i
+P_{<q}(n_1,n_2),
\]
where $P_{<q}$ has total degree less than $q$.
Taking $n_1=n_2=n$ yields
\begin{equation}\label{formula0}
\ell_R\!\left(
(S_{2n}/E_1^nE_2^n)\otimes_RN
\right)
=
\frac{n^q}{q!}
\sum_{i=0}^{q}\binom qi
e_{\mathrm{BR}}(E_1^{[q-i]},E_2^{[i]};N)
+O(n^{q-1}).
\end{equation}
On the other hand, the definition of multiplicity in
degree two gives
\begin{equation}\label{formula1}
\ell_R\!\left(
(S_{2n}/(E_1E_2)^n)\otimes_RN
\right)
=
\frac{\widetilde e_{\mathrm{BR}}(E_1E_2;N)}{q!}\,n^q
+O(n^{q-1}).
\end{equation}
Since $(E_1E_2)^n=E_1^nE_2^n$ inside $S$, the left-hand
sides of \eqref{formula0} and \eqref{formula1} agree.
Comparing the coefficients of $n^q$ proves the formula.
\end{proof}

\begin{cor}[Mixed product formula in arbitrary degrees]\label{mixed}
Fix positive integers $e_1,e_2$, and let
$E_i\subseteq S_{e_i}=\operatorname{Sym}_R^{e_i}(F)$ have finite
colength for $i=1,2$. Then, for all positive integers $r,s$,
\begin{equation}\label{formula}
\widetilde e_{\rm BR}(E_1^rE_2^s;N)
=\sum_{i=0}^{q}\binom qi
\widetilde e_{\rm BR}\left(E_1^{[q-i]},E_2^{[i]};N\right)
 r^{q-i}s^i.
\end{equation}
When $e_1=e_2=1$, this becomes
\begin{equation}\label{formula2}
\widetilde e_{\rm BR}(E_1^rE_2^s;N)
=\sum_{i=0}^{q}\binom qi
 e_{\rm BR}\left(E_1^{[q-i]},E_2^{[i]};N\right)
 r^{q-i}s^i.
\end{equation}
\end{cor}

\begin{proof}
Write the eventual mixed Hilbert polynomial as $P(X,Y)$. Its
homogeneous component of degree $q$ is
\[
\sum_{i=0}^{q}\frac{\widetilde e_{\rm BR}
(E_1^{[q-i]},E_2^{[i]};N)}{(q-i)!\,i!}X^{q-i}Y^i.
\]
For fixed positive integers $r,s$, substituting $(X,Y)=(rn,sn)$
shows that the leading coefficient of $P(rn,sn)$ is
\[
\frac{1}{q!}\sum_{i=0}^{q}\binom qi
\widetilde e_{\rm BR}(E_1^{[q-i]},E_2^{[i]};N)r^{q-i}s^i.
\]
The same function is the Hilbert function defining
$\widetilde e_{\rm BR}(E_1^rE_2^s;N)$, since
$(E_1^rE_2^s)^n=E_1^{rn}E_2^{sn}$. Comparing leading coefficients
proves \eqref{formula}. If $e_1=e_2=1$, the arbitrary-degree
coefficients are the ordinary mixed Buchsbaum--Rim multiplicities,
which gives \eqref{formula2}.
\end{proof}

\begin{cor}[Product formula for ideals]\label{cor:ideal-product}
Let $I,J$ be $\mathfrak m$-primary ideals of $R$. In the special case
$p=1$ (so that $F=R$ and $q=d$), the preceding formula becomes
\[
e(I^rJ^s;N)=
\sum_{i=0}^{d}\binom di
e\left(I^{[d-i]},J^{[i]};N\right)r^{d-i}s^i
\qquad(r,s\geq0).
\]
In particular,
\[
e(IJ;N)=\sum_{i=0}^{d}\binom di
e\left(I^{[d-i]},J^{[i]};N\right).
\]
\end{cor}

\begin{proof}
When $p=1$, $S=\operatorname{Sym}_R(R)=R[t]$ and each graded
component $S_n$ is a free $R$-module of rank one.  Thus the
Buchsbaum--Rim multiplicities and the mixed Buchsbaum--Rim
multiplicities specialize to the usual Hilbert--Samuel mixed
multiplicities of $I$ and $J$.  The identities follow from
\eqref{formula2}.
\end{proof}

\begin{lem}\label{lem:mixed-br-basic-properties}

\begin{enumerate}
\item For finite-colength submodules $E_1,E_2\subseteq F$ and
positive integers $u,v$, regard $E_1^u$ and $E_2^v$ as submodules of
$S_u$ and $S_v$, respectively. Then, for $0\leq i\leq q$,
\[
\widetilde e_{\rm BR}\!\left((E_1^u)^{[q-i]},(E_2^v)^{[i]}\right)
=u^{q-i}v^i e_{\rm BR}\!\left(E_1^{[q-i]},E_2^{[i]}\right).
\]
\item For every finite-colength submodule $E\subseteq F$ and
$0\leq i\leq q$,
\[
e_{\rm BR}\!\left(E^{[q-i]},E^{[i]}\right)=e_{\rm BR}(E).
\]
\item If $K_1\subseteq M_1$ and $K_2\subseteq M_2$ are
finite-colength submodules of $F$, with $K_j$ a reduction of $M_j$
for $j=1,2$, then
\[
e_{\rm BR}\!\left(K_1^{[q-i]},K_2^{[i]}\right)
=e_{\rm BR}\!\left(M_1^{[q-i]},M_2^{[i]}\right)
\qquad(0\leq i\leq q).
\]
\end{enumerate}
\end{lem}
\begin{proof}
\smallskip
\noindent\textup{(i)}
For every pair of positive integers $r,s$, Corollary~\ref{mixed}
and the identity
$(E_1^u)^r(E_2^v)^s=E_1^{ur}E_2^{vs}$ give
\[
\begin{aligned}
\sum_{i=0}^{q}\binom{q}{i}
\widetilde e_{\rm BR}
\left((E_1^u)^{[q-i]},(E_2^v)^{[i]}\right)
r^{q-i}s^i&= \widetilde e_{\rm BR}(E_1^{ur}E_2^{vs})\\
=&
\sum_{i=0}^{q}\binom qi
e_{\rm BR}\left(E_1^{[q-i]},E_2^{[i]}\right)
u^{q-i}v^i r^{q-i}s^i.
\end{aligned}
\]
Since this equality holds for all positive integers $r,s$,
it is a polynomial identity. Comparing coefficients of
$r^{q-i}s^i$ proves \textup{(i)}.

\smallskip
\noindent\textup{(ii)}
By Corollary~\ref{mixed} and Remark~\ref{remark1},
\[
\begin{aligned}
\sum_{i=0}^{q}\binom qi
e_{\rm BR}\left(E_1^{[q-i]},E_1^{[i]}\right)
r^{q-i}s^i
&=\widetilde e_{\rm BR}(E_1^{r+s})\\
&=e_{\rm BR}(E_1)(r+s)^q\\
&=\sum_{i=0}^{q}\binom qi
e_{\rm BR}(E_1)r^{q-i}s^i.
\end{aligned}
\]
Comparing coefficients proves \textup{(ii)}.

\smallskip
\noindent\textup{(iii)}
We first establish reduction invariance for multiplicities
in arbitrary degree. Let $A\subseteq B\subseteq S_e$ have
finite colength, and suppose that $A$ is a reduction of $B$.
Choose $h\geq1$ such that
\(
B^n=A^{n-h}B^h
\,\,(n\geq h).
\)
Since $A^n\subseteq B^n\subseteq A^{n-h}S_{eh}$, we obtain
\[
\ell_R(S_{en}/A^{n-h}S_{eh})
\leq
\ell_R(S_{en}/B^n)
\leq
\ell_R(S_{en}/A^n).
\]
Applying Lemma~\ref{lem:br-degree-shift} to $A$, with
coefficient module $R$ and index $n-h$, gives
\[
\ell_R(S_{en}/A^{n-h}S_{eh})
=
\ell_R(S_{en}/A^n)+O(n^{q-1}).
\]
Thus the two outer terms have the same normalized leading
coefficient. Dividing by $n^q/q!$ and taking limits yields
\[
\widetilde e_{\rm BR}(A)=\widetilde e_{\rm BR}(B).
\]

Now fix positive integers $r,s$. Since $K_1$ is a reduction
of $M_1$ and $K_2$ is a reduction of $M_2$, for $n\gg0$ we have
\[
M_1^{r(n+1)}=K_1^rM_1^{rn},
\qquad
M_2^{s(n+1)}=K_2^sM_2^{sn}.
\]
Consequently,
\[
(M_1^rM_2^s)^{n+1}
=(K_1^rK_2^s)(M_1^rM_2^s)^n.
\]
Hence $K_1^rK_2^s$ is a reduction of $M_1^rM_2^s$, both
regarded as submodules of $S_{r+s}$.

These modules have finite colength: choosing $a,b\geq1$
with $\mathfrak m^aF\subseteq K_1$ and
$\mathfrak m^bF\subseteq K_2$, we have
\[
\mathfrak m^{ar+bs}S_{r+s}
\subseteq K_1^rK_2^s
\subseteq M_1^rM_2^s.
\]
The reduction invariance proved above therefore gives
\[
\widetilde e_{\rm BR}(K_1^rK_2^s)
=\widetilde e_{\rm BR}(M_1^rM_2^s).
\]
By Corollary~\ref{mixed}, this equality becomes
\[
\begin{aligned}
&\sum_{i=0}^{q}\binom qi
e_{\rm BR}\left(K_1^{[q-i]},K_2^{[i]}\right)r^{q-i}s^i=
\sum_{i=0}^{q}\binom qi
e_{\rm BR}\left(M_1^{[q-i]},M_2^{[i]}\right)r^{q-i}s^i.
\end{aligned}
\]
Since it holds for all positive integers $r,s$, comparison
of coefficients proves \textup{(iii)}.
\end{proof}

\section{Minkowski-type inequalities for modules}
\label{sec:minkowski}

First, we consider a natural analogue of one of the Minkowski
inequalities, adapting \eqref{CM1} to finite-colength submodules.
For $E_1,E_2\subseteq F$ of finite colength, the coefficient-module version
of the proposed power inequality is
\begin{equation}\label{CM3N}
e_{\rm BR}\left(E_1^{[q-i]},E_2^{[i]};N\right)^q
\leq e_{\rm BR}(E_1;N)^{q-i}e_{\rm BR}(E_2;N)^i,
\qquad 1\leq i\leq q-1.
\end{equation}

The following examples show that this inequality fails in general, in dimensions one, two, and three.
\subsection{Failure of the coefficientwise power inequalities}
\begin{ex}\label{ex:mixed-br-teissier-failure}
{\rm The classical quadratic Teissier inequality does not extend
to mixed Buchsbaum--Rim multiplicities under the finite-colength
hypothesis alone.

Let $R=k[[x]]$, let $F=R^2$, and set $E_1=xF$ and $E_2=x^2F$.
Then $d+p-1=2$. For $S=\operatorname{Sym}_R(F)$, one has
\[
E_1^{n_1}E_2^{n_2}
=x^{n_1+2n_2}S_{n_1+n_2},
\]
and consequently
\[
\ell_R(S_{n_1+n_2}/E_1^{n_1}E_2^{n_2})
=(n_1+2n_2)(n_1+n_2+1).
\]
Comparing the quadratic terms gives
\[
e_{\mathrm{BR}}(E_1)=2,\qquad
e_{\mathrm{BR}}(E_2)=4,\qquad
e_{\mathrm{BR}}(E_1^{[1]},E_2^{[1]})=3.
\]
Hence
\[
e_{\mathrm{BR}}(E_1^{[1]},E_2^{[1]})^2
=9>8
=e_{\mathrm{BR}}(E_1)e_{\mathrm{BR}}(E_2).
\]}
\end{ex}

\begin{ex}\label{ex:failure-teissier-modules}
{\rm The direct analogue of Teissier's inequality fails for mixed
Buchsbaum--Rim multiplicities, even over a regular local ring.

Let $R=k[[x,y,z]]$, with maximal ideal $\mathfrak m=(x,y,z)$,
and let $F=R^2$ and $N=R$. Set
\[
E_1=\mathfrak mF,\qquad E_2=\mathfrak m^2F.
\]
Both modules have finite colength in $F$, and
$q=\dim R+\operatorname{rank}F-1=4$.
Writing $S=\operatorname{Sym}_R(F)$, we have
\[
E_1^{n_1}E_2^{n_2}
=\mathfrak m^{n_1+2n_2}S_{n_1+n_2}.
\]
Since $S_{n_1+n_2}$ is free of rank $n_1+n_2+1$ and
$\ell_R(R/\mathfrak m^t)=\binom{t+2}{3}$, it follows that
\[
h_R(n_1,n_2;R)
=(n_1+n_2+1)\binom{n_1+2n_2+2}{3}.
\]
Its homogeneous component of total degree four is
\[
\frac{(n_1+n_2)(n_1+2n_2)^3}{6}
=
\frac{
n_1^4+7n_1^3n_2+18n_1^2n_2^2
+20n_1n_2^3+8n_2^4
}{6}.
\]
Comparing this expression with
\(
\sum_{i=0}^{4}
\frac{
e_{\mathrm{BR}}(E_1^{[4-i]},E_2^{[i]})
}{(4-i)!\,i!}\,
n_1^{4-i}n_2^i,
\)
we obtain
\[
e_{\mathrm{BR}}(E_1)=4,\qquad
e_{\mathrm{BR}}(E_2)=32,\qquad
e_{\mathrm{BR}}(E_1^{[3]},E_2^{[1]})=7.
\]
Consequently,
\[
e_{\mathrm{BR}}(E_1^{[3]},E_2^{[1]})^4
=7^4=2401
>
2048=4^3\cdot32
=e_{\mathrm{BR}}(E_1)^3e_{\mathrm{BR}}(E_2).
\]
Thus the inequality
\[
e_{\mathrm{BR}}(E_1^{[q-i]},E_2^{[i]})^q
\leq
e_{\mathrm{BR}}(E_1)^{q-i}e_{\mathrm{BR}}(E_2)^i
\]
does not hold in general for submodules of finite colength.}
\end{ex}

\begin{ex}\label{ex:diagonal-modules-teissier-failure}
{\rm The analogue of Teissier's inequality fails even for diagonal
modules defined by distinct integrally closed primary ideals
over a two-dimensional regular local ring.

Let $R=k[[x,y]]$, with maximal ideal $\mathfrak m=(x,y)$,
and let $F=Rf_1\oplus Rf_2$. Set
\[
E_1=\mathfrak m f_1\oplus\mathfrak m^2f_2,
\qquad
E_2=\mathfrak m^3f_1\oplus\mathfrak m^4f_2.
\]
The four coordinate ideals are distinct, integrally closed,
and $\mathfrak m$-primary. Moreover,
$q=\dim R+\operatorname{rank}F-1=3$.

Identify $S=\operatorname{Sym}_R(F)$ with $R[t_1,t_2]$.
For $0\leq k\leq n_1+n_2$, the coefficient ideal of
$t_1^kt_2^{n_1+n_2-k}$ in $E_1^{n_1}E_2^{n_2}$ is
$\mathfrak m^{2n_1+4n_2-k}$. Indeed, every admissible
decomposition $k=i+j$ contributes the ideal
\[
\mathfrak m^{i+2(n_1-i)+3j+4(n_2-j)}
=\mathfrak m^{2n_1+4n_2-k}.
\]
Consequently,
\[
h_R(n_1,n_2;R)
=
\sum_{k=0}^{n_1+n_2}
\ell_R(R/\mathfrak m^{2n_1+4n_2-k})
=
\sum_{k=0}^{n_1+n_2}
\binom{2n_1+4n_2-k+1}{2}.
\]
Its homogeneous component of total degree three is
\[
\frac{(2n_1+4n_2)^3-(n_1+3n_2)^3}{6}
=
\frac{
7n_1^3+39n_1^2n_2+69n_1n_2^2+37n_2^3
}{6}.
\]
Comparing coefficients with
\[
\sum_{i=0}^{3}
\frac{e_{\mathrm{BR}}(E_1^{[3-i]},E_2^{[i]})}
{(3-i)!\,i!}\,n_1^{3-i}n_2^i,
\]
we obtain
\[
\begin{aligned}
e_{\mathrm{BR}}(E_1)&=7,
&
e_{\mathrm{BR}}(E_2)&=37,\\
e_{\mathrm{BR}}(E_1^{[2]},E_2^{[1]})&=13,
&
e_{\mathrm{BR}}(E_1^{[1]},E_2^{[2]})&=23.
\end{aligned}
\]
Thus both inequalities fail:
\[
13^3=2197>1813=7^2\cdot37,
\qquad
23^3=12167>9583=7\cdot37^2.
\]
In particular, requiring the coordinate ideals to be distinct
and integrally closed does not suffice to guarantee
\[
e_{\mathrm{BR}}(E_1^{[q-i]},E_2^{[i]})^q
\leq
e_{\mathrm{BR}}(E_1)^{q-i}e_{\mathrm{BR}}(E_2)^i.
\]}
\end{ex}

\begin{ex}\label{ex:non-diagonal-teissier-failure}
{\rm The analogue of Teissier's inequality also fails for
finite-colength submodules that cannot be diagonalized
by a change of basis of the ambient free module.

Let $R=k[[x,y]]$, with maximal ideal $\mathfrak m=(x,y)$,
let $F=Rf_1\oplus Rf_2$, and take $N=R$. Consider
\[
E_1=
\operatorname{Im}
\begin{pmatrix}
x&y&0\\
0&x&y
\end{pmatrix},
\qquad
E_2=
\operatorname{Im}
\begin{pmatrix}
x^2&y^2&0\\
0&x^2&y^2
\end{pmatrix}.
\]
Their maximal-minor ideals are
$(x^2,xy,y^2)$ and $(x^4,x^2y^2,y^4)$, respectively.
Both are $\mathfrak m$-primary, so
$\ell_R(F/E_j)<\infty$ for $j=1,2$.

Each module has exactly three minimal generators.
Indeed, the displayed generators have linearly independent
initial forms in degrees one and two, respectively.
Moreover, $E_j\subseteq\mathfrak mF$.
If either module were diagonal in some basis of $F$, its
two coordinate ideals would therefore be proper
$\mathfrak m$-primary ideals. Each would require at least
two generators, contradicting $\mu_R(E_j)=3$.
Thus neither module is diagonalizable.

We next determine their reductions. For elements $a,b\in R$,
write
\[
M=(a,b)F,\qquad
E=R(af_1)+R(bf_1+af_2)+R(bf_2).
\]
Inside $S=\operatorname{Sym}_R(F)=R[t_1,t_2]$, put
\(
A=at_1,\,\, B=bt_1,\,\, C=at_2,\,\, D=bt_2.
\)
Then $M$ is generated by $A,B,C,D$, whereas $E$ is
generated by $A,B+C,D$. Since $BC=AD$, we have
\[
\begin{aligned}
BC&\in EM,\\
B^2&=B(B+C)-AD\in EM,\\
C^2&=C(B+C)-AD\in EM.
\end{aligned}
\]
Every remaining quadratic product of the generators
contains $A$ or $D$. Hence $M^2=EM$, proving that
$E$ is a reduction of $M$.

Applying this observation with $(a,b)=(x,y)$ shows that
$E_1$ is a reduction of $\mathfrak mF$.
With $(a,b)=(x^2,y^2)$, it shows that $E_2$ is a reduction
of $(x^2,y^2)F$.
Furthermore,
\[
(\mathfrak m^2)^2=(x^2,y^2)\mathfrak m^2,
\]
so $(x^2,y^2)F$ is a reduction of $\mathfrak m^2F$.
By transitivity, $E_2$ is a reduction of $\mathfrak m^2F$.

By invariance of mixed Buchsbaum--Rim multiplicities under
reductions in each argument, it remains to compute the
multiplicities of $\mathfrak mF$ and $\mathfrak m^2F$.
Their length function is
\[
\ell_R\!\left(
S_{n_1+n_2}/
(\mathfrak mF)^{n_1}(\mathfrak m^2F)^{n_2}
\right)
=
(n_1+n_2+1)\binom{n_1+2n_2+1}{2}.
\]
Its homogeneous component of degree three is
\[
\frac{(n_1+n_2)(n_1+2n_2)^2}{2}
=
\frac{n_1^3+5n_1^2n_2+8n_1n_2^2+4n_2^3}{2}.
\]
Comparing normalized coefficients yields
\[
\begin{aligned}
e_{\mathrm{BR}}(E_1)&=3,
&
e_{\mathrm{BR}}(E_2)&=12,\\
e_{\mathrm{BR}}(E_1^{[2]},E_2^{[1]})&=5,
&
e_{\mathrm{BR}}(E_1^{[1]},E_2^{[2]})&=8.
\end{aligned}
\]
Since $q=\dim R+\operatorname{rank}F-1=3$, both proposed
inequalities fail:
\[
e_{\mathrm{BR}}(E_1^{[2]},E_2^{[1]})^3
=125>108
=e_{\mathrm{BR}}(E_1)^2e_{\mathrm{BR}}(E_2),
\]
and
\[
e_{\mathrm{BR}}(E_1^{[1]},E_2^{[2]})^3
=512>432
=e_{\mathrm{BR}}(E_1)e_{\mathrm{BR}}(E_2)^2.
\]}
\end{ex}

\subsection{Log-convexity and the product-root inequality}

We next consider the proposed log-convexity analogue of
\eqref{CM4}. For finite-colength submodules $E_1,E_2\subseteq F$,
the proposed inequality is
$$
\begin{aligned}
e_{\rm BR}\left(E_1^{[q-i]},E_2^{[i]};N\right)^2
&\leq e_{\rm BR}\left(E_1^{[q-i-1]},E_2^{[i+1]};N\right)\cdot e_{\rm BR}\left(E_1^{[q-i+1]},E_2^{[i-1]};N\right),
\,\, 1\leq i\leq q-1.
\end{aligned}
$$

\subsubsection{Counterexamples to log-convexity and the product-root inequality}

The examples below show that neither log-convexity nor the product-root inequality holds in general.

\begin{ex}\label{ex:failure-log-convexity-dim-three}
{\rm Mixed Buchsbaum--Rim multiplicities need not satisfy the
Teissier log-convexity inequalities, even over a regular local
ring of dimension three.

Let $R=k[[x,y,z]]$, with maximal ideal $\mathfrak m=(x,y,z)$,
and set $F=R^2$, $N=R$, $E_1=\mathfrak mF$, and
$E_2=\mathfrak m^2F$. Both modules have finite colength in $F$.
Here $d=3$, $p=2$, and $q=d+p-1=4$.

Writing $S=\operatorname{Sym}_R(F)=R[t_1,t_2]$, we have
\[
E_1^{n_1}E_2^{n_2}
=\mathfrak m^{n_1+2n_2}S_{n_1+n_2}.
\]
Since $S_{n_1+n_2}$ is free of rank $n_1+n_2+1$ and
$\ell_R(R/\mathfrak m^a)=\binom{a+2}{3}$, it follows that
\[
h_R(n_1,n_2;R)
=
\ell_R\left(
\frac{S_{n_1+n_2}}{E_1^{n_1}E_2^{n_2}}
\right)
=
(n_1+n_2+1)\binom{n_1+2n_2+2}{3}.
\]
Its homogeneous component of degree four is
\[
\begin{aligned}
[h_R]_4
&=\frac{(n_1+n_2)(n_1+2n_2)^3}{6}\\
&=\frac{
n_1^4+7n_1^3n_2+18n_1^2n_2^2
+20n_1n_2^3+8n_2^4}{6}.
\end{aligned}
\]
Set
\[
c_i=e_{\rm BR}\left(E_1^{[4-i]},E_2^{[i]};R\right),
\qquad 0\leq i\leq4.
\]
By definition,
\[
[h_R]_4
=\sum_{i=0}^4\frac{c_i}{(4-i)!\,i!}n_1^{4-i}n_2^i.
\]
Comparing coefficients gives
\[
(c_0,c_1,c_2,c_3,c_4)=(4,7,12,20,32).
\]
Consequently,
\[
\begin{aligned}
c_1^2&=49>48=c_0c_2,\\
c_2^2&=144>140=c_1c_3,\\
c_3^2&=400>384=c_2c_4.
\end{aligned}
\]
Thus the inequalities $c_i^2\leq c_{i-1}c_{i+1}$ fail for
every $i=1,2,3$, even for modules of the form $IF$ and $JF$
with $I,J$ $\mathfrak m$-primary.}
\end{ex}

\begin{ex}\label{ex:failure-minkowski-modules}
{\rm The following example shows that both the inequalities
\[
e_{\rm BR}\left(E_1^{[q-i]},E_2^{[i]};N\right)
e_{\rm BR}\left(E_1^{[i]},E_2^{[q-i]};N\right)
\leq e_{\rm BR}(E_1;N)e_{\rm BR}(E_2;N)
\]
and
\[
\widetilde e_{\rm BR}(E_1E_2;N)^{1/q}
\leq e_{\rm BR}(E_1;N)^{1/q}
   +e_{\rm BR}(E_2;N)^{1/q},
\qquad q=d+p-1,
\]
can fail, even over a regular local ring of dimension three.

Let $R=k[[x,y,z]]$, with maximal ideal $\mathfrak m=(x,y,z)$,
and set
\(
F=R^2,\,\, N=R,\,\,
E_1=\mathfrak mF,\,\, E_2=\mathfrak m^2F.
\)
Both modules have finite colength in $F$, and $q=4$.
Writing $S=\operatorname{Sym}_R(F)=R[t_1,t_2]$, we obtain
\(
E_1^{n_1}E_2^{n_2}
=\mathfrak m^{n_1+2n_2}S_{n_1+n_2}.
\)
Therefore,
\[
h_R(n_1,n_2;R)
=(n_1+n_2+1)\binom{n_1+2n_2+2}{3},
\]
whose homogeneous component of degree four is
\[
[h_R]_4
=\frac{
n_1^4+7n_1^3n_2+18n_1^2n_2^2
+20n_1n_2^3+8n_2^4}{6}.
\]
Setting
\[
e_i=e_{\rm BR}\left(E_1^{[4-i]},E_2^{[i]};R\right),
\qquad 0\leq i\leq4,
\]
and comparing coefficients with
$\sum_{i=0}^4 e_i n_1^{4-i}n_2^i/((4-i)!\,i!)$ gives
\[
(e_0,e_1,e_2,e_3,e_4)=(4,7,12,20,32).
\]
In particular,
\[
e_1e_3=140>128=e_0e_4,
\qquad
e_2^2=144>128=e_0e_4.
\]
Thus the first proposed inequality fails for every
$i=1,2,3$.

Moreover, the product formula yields
\[
\widetilde e_{\rm BR}(E_1E_2;R)
=\sum_{i=0}^{4}\binom4i e_i
=4+28+72+80+32
=216.
\]
By the strict concavity of $t\mapsto t^{3/4}$,
\[
\frac{1+2^{3/4}}2
<
\left(\frac32\right)^{3/4}.
\]
Hence
\[
\left(4^{1/4}+32^{1/4}\right)^4
=4\left(1+2^{3/4}\right)^4
<64\left(\frac32\right)^3
=216.
\]
Consequently,
\[
\widetilde e_{\rm BR}(E_1E_2;R)^{1/4}
>
e_{\rm BR}(E_1;R)^{1/4}
+e_{\rm BR}(E_2;R)^{1/4},
\]
which disproves the proposed Minkowski-type inequality.}
\end{ex}

\subsubsection{An equality criterion for the product-root inequality}

\begin{prop}\label{prop:root-equality-criterion}
Let $E_1,E_2\subsetneq F$ have finite colength.
The following conditions are equivalent:
\begin{enumerate}
\item[(a)]
All mixed multiplicities
$e_{\rm BR}(E_1^{[q-i]},E_2^{[i]})$ are positive,
and there exist positive integers $r,s$ such that
\[
\frac{
e_{\rm BR}\left(E_1^{[q-i]},E_2^{[i]}\right)
}{
e_{\rm BR}\left(E_1^{[q-i+1]},E_2^{[i-1]}\right)
}
=\frac rs
\qquad(1\leq i\leq q).
\]
\item[(b)]
For every $0\leq i\leq q$,
\[
e_{\rm BR}\left(E_1^{[q-i]},E_2^{[i]}\right)
=
e_{\rm BR}(E_1)^{(q-i)/q}
e_{\rm BR}(E_2)^{i/q}.
\]
\end{enumerate}
Either condition implies
\[
\widetilde e_{\rm BR}(E_1E_2)^{1/q}
=
e_{\rm BR}(E_1)^{1/q}+e_{\rm BR}(E_2)^{1/q}.
\]
Conversely, if
\[
e_{\rm BR}\left(E_1^{[q-i]},E_2^{[i]}\right)
\leq
e_{\rm BR}(E_1)^{(q-i)/q}
e_{\rm BR}(E_2)^{i/q}
\qquad(1\leq i\leq q-1),
\]
then this equality implies conditions \textup{(a)} and
\textup{(b)}.
\end{prop}

\begin{proof}
Suppose \textup{(a)} holds. Iterating the ratio identities gives
\[
e_{\rm BR}\left(E_1^{[q-i]},E_2^{[i]}\right)
=e_{\rm BR}(E_1)\left(\frac rs\right)^i
\qquad(0\leq i\leq q).
\]
Taking $i=q$, we obtain
\[
e_{\rm BR}(E_2)
=e_{\rm BR}(E_1)\left(\frac rs\right)^q.
\]
Since $r/s>0$ and $e_{\rm BR}(E_1)>0$, it follows that
\[
\frac rs
=
\left(\frac{e_{\rm BR}(E_2)}{e_{\rm BR}(E_1)}\right)^{1/q}.
\]
Substituting this expression into the first identity yields,
for every $0\leq i\leq q$,
\[
\begin{aligned}
e_{\rm BR}\left(E_1^{[q-i]},E_2^{[i]}\right)
&=e_{\rm BR}(E_1)
\left(\frac{e_{\rm BR}(E_2)}{e_{\rm BR}(E_1)}\right)^{i/q}\\
&=e_{\rm BR}(E_1)^{(q-i)/q}
  e_{\rm BR}(E_2)^{i/q},
\end{aligned}
\]
which proves \textup{(b)}.

Conversely, suppose \textup{(b)} holds. Since
$e_{\rm BR}(E_1)$ and $e_{\rm BR}(E_2)$ are positive,
all the mixed multiplicities are positive. Moreover, for
$1\leq i\leq q$,
\[
\begin{aligned}
\frac{
e_{\rm BR}\left(E_1^{[q-i]},E_2^{[i]}\right)
}{
e_{\rm BR}\left(E_1^{[q-i+1]},E_2^{[i-1]}\right)
}
&=
\frac{
e_{\rm BR}(E_1)^{(q-i)/q}e_{\rm BR}(E_2)^{i/q}
}{
e_{\rm BR}(E_1)^{(q-i+1)/q}e_{\rm BR}(E_2)^{(i-1)/q}
}\\
&=
\left(\frac{e_{\rm BR}(E_2)}{e_{\rm BR}(E_1)}\right)^{1/q}.
\end{aligned}
\]
Thus all consecutive ratios coincide. Taking $i=1$ shows
that their common value is
\(
\frac{
e_{\rm BR}\left(E_1^{[q-1]},E_2^{[1]}\right)
}{
e_{\rm BR}(E_1)
}.
\) As both the numerator and the denominator are positive
integers, this value can be written as $r/s$ with
$r,s\geq1$, proving \textup{(a)}.

We next prove the asserted equality. Under \textup{(b)},
Lemma~\ref{lem:mixed-product-formula} and the binomial theorem give
\[
\begin{aligned}
\widetilde e_{\rm BR}(E_1E_2)
&=\sum_{i=0}^{q}\binom qi
e_{\rm BR}\left(E_1^{[q-i]},E_2^{[i]}\right)\\
&=\sum_{i=0}^{q}\binom qi
e_{\rm BR}(E_1)^{(q-i)/q}
e_{\rm BR}(E_2)^{i/q}\\
&=\left(
e_{\rm BR}(E_1)^{1/q}
+e_{\rm BR}(E_2)^{1/q}
\right)^q.
\end{aligned}
\]
Taking $q$-th roots yields the desired conclusion.

Finally, assume that the additional coefficient bounds
hold and that
\[
\widetilde e_{\rm BR}(E_1E_2)^{1/q}
=
e_{\rm BR}(E_1)^{1/q}+e_{\rm BR}(E_2)^{1/q}.
\]
Raising both sides to the $q$-th power and applying
Lemma~\ref{lem:mixed-product-formula} and the binomial theorem, we obtain
\[
\sum_{i=0}^{q}\binom qi
\left[
e_{\rm BR}(E_1)^{(q-i)/q}
e_{\rm BR}(E_2)^{i/q}
-
e_{\rm BR}\left(E_1^{[q-i]},E_2^{[i]}\right)
\right]=0.
\]
The bracketed expressions vanish for $i=0,q$ and are
nonnegative for $1\leq i\leq q-1$ by assumption.
Since every binomial coefficient is positive, each
bracketed expression must vanish. This proves
\textup{(b)}, and hence also \textup{(a)}.
\end{proof}

\subsection{Additional bounds for product multiplicities}

\subsubsection{An aggregate bound via annihilator ideals}\label{subsec:annihilator-bound}

The preceding examples show that the coefficientwise estimate
\eqref{CM3N} can fail. A related estimate does hold after summing
the mixed coefficients: it bounds the product multiplicity in terms
of the annihilator ideals of $F/E_1$ and $F/E_2$.

\begin{prop}\label{prop:general-minkowski-br}

Let $E_1,E_2\subsetneq F$ have finite colength, and set
\(
I_j=\operatorname{Ann}_R(F/E_j)=(E_j:_RF)
\,\,(j=1,2).
\)
Then, for all integers $r,s\geq1$,
\[
\begin{aligned}
&\sum_{i=0}^{q}\binom qi
e_{\mathrm{BR}}(E_1^{[q-i]},E_2^{[i]};N)
r^{q-i}s^i\leq \binom qd(r+s)^{p-1}
\left(
r\,e(I_1;N)^{1/d}
+s\,e(I_2;N)^{1/d}
\right)^d.
\end{aligned}
\]
\end{prop}

\begin{proof}
The ideals $I_1,I_2$ are $\mathfrak m$-primary and satisfy
$I_jF\subseteq E_j$. Hence, writing
$S=\operatorname{Sym}_R(F)$, we obtain
\[
(I_1^rI_2^s)^nS_{(r+s)n}
\subseteq E_1^{rn}E_2^{sn}.
\]
By right exactness of tensor products,
\[
\begin{aligned}
&\ell_R\!\left(
(S_{(r+s)n}/E_1^{rn}E_2^{sn})\otimes_RN
\right)\leq
\binom{(r+s)n+p-1}{p-1}
\ell_R\!\left(N/(I_1^rI_2^s)^nN\right).
\end{aligned}
\]
Multiplying by $q!/n^q$ and taking limits yields
\[
\widetilde e_{\mathrm{BR}}(E_1^rE_2^s;N)
\leq
\binom qd(r+s)^{p-1}e(I_1^rI_2^s;N).
\]
The product formula of Corollary~\ref{mixed} gives
\[
\widetilde e_{\mathrm{BR}}(E_1^rE_2^s;N)
=
\sum_{i=0}^{q}\binom qi
e_{\mathrm{BR}}(E_1^{[q-i]},E_2^{[i]};N)
r^{q-i}s^i,
\]
proving the first inequality.

For the second inequality, apply Minkowski's inequality
for ideals with coefficients in $N$:
\[
e(I_1^rI_2^s;N)^{1/d}
\leq
r\,e(I_1;N)^{1/d}+s\,e(I_2;N)^{1/d}.
\]
For completeness, this module version follows from the
classical inequality \cite{RS,Verma}, associativity, and
H\"older's inequality. Indeed, let
\[
\Lambda=\{\mathfrak p\in\operatorname{Supp}_RN:
                 \dim(R/\mathfrak p)=d\},
\qquad
w_{\mathfrak p}=\ell_{R_{\mathfrak p}}(N_{\mathfrak p}),
\]
and put $a_{\mathfrak p}=e(I_1;R/\mathfrak p)^{1/d}$
and $b_{\mathfrak p}=e(I_2;R/\mathfrak p)^{1/d}$.
Associativity and the classical inequality give
\[
e(I_1^rI_2^s;N)^{1/d}
\leq
\left(
\sum_{\mathfrak p\in\Lambda}
w_{\mathfrak p}(ra_{\mathfrak p}+sb_{\mathfrak p})^d
\right)^{1/d}.
\]
The weighted $\ell^d$ triangle inequality bounds this by
\[
r\left(\sum_{\mathfrak p\in\Lambda}
w_{\mathfrak p}a_{\mathfrak p}^d\right)^{1/d}
+
s\left(\sum_{\mathfrak p\in\Lambda}
w_{\mathfrak p}b_{\mathfrak p}^d\right)^{1/d}
=
r\,e(I_1;N)^{1/d}+s\,e(I_2;N)^{1/d}.
\]
This proves the second inequality.
\end{proof}

\begin{rem}
For $p=1$, Proposition~\ref{prop:general-minkowski-br}
recovers the classical Minkowski inequality for ideals.
Indeed, identify $F$ with $R$. Then $E_1,E_2$ are
$\mathfrak m$-primary ideals,
$I_j=(E_j:_RR)=E_j$, and $q=d$.
Buchsbaum--Rim multiplicities become Hilbert--Samuel
multiplicities, so the proposition gives
\[
\sum_{i=0}^{d}\binom di
e(E_1^{[d-i]},E_2^{[i]};N)r^{d-i}s^i
\leq
\left(
r\,e(E_1;N)^{1/d}
+s\,e(E_2;N)^{1/d}
\right)^d.
\]
By the product formula for mixed multiplicities, the
left-hand side equals $e(E_1^rE_2^s;N)$. Therefore
\[
e(E_1^rE_2^s;N)^{1/d}
\leq
r\,e(E_1;N)^{1/d}+s\,e(E_2;N)^{1/d}.
\]
Taking $r=s=1$ yields
\[
e(E_1E_2;N)^{1/d}
\leq e(E_1;N)^{1/d}+e(E_2;N)^{1/d},
\]
which, for $N=R$, is the classical Minkowski inequality.
\end{rem}

\subsubsection{A comparison bound via the intersection}

The following result bounds the product multiplicity in terms of the
intersection of the two modules. It is a comparison estimate and does
not assert log-convexity.

\begin{prop}\label{prop_br_bounds}
Let $E_1,E_2\subseteq F$ have finite colength.
Set $K=E_1\cap E_2$. Then
\[
\begin{aligned}
\widetilde e_{\rm BR}(E_1E_2;N)
&=
\sum_{i=0}^{q}\binom{q}{i}
e_{\rm BR}\left(E_1^{[q-i]},E_2^{[i]};N\right)\\
&\leq 2^q e_{\rm BR}(K;N).
\end{aligned}
\]
In particular, if $E_1\subseteq E_2$ or $E_2\subseteq E_1$,
then
\[
\widetilde e_{\rm BR}(E_1E_2;N)
\leq
2^q\max\{e_{\rm BR}(E_1;N),e_{\rm BR}(E_2;N)\}.
\]
Equality holds in both estimates when $E_1=E_2$.
\end{prop}

\begin{proof}
Write $S=\operatorname{Sym}_R(F)$.
The natural injection \(
F/K\longrightarrow F/E_1\oplus F/E_2
\)
shows that $K$ has finite colength in $F$. Since $K\subseteq E_1$ and $K\subseteq E_2$, for every
$n\geq1$ we have
\[
K^{2n}\subseteq E_1^nE_2^n=(E_1E_2)^n
\subseteq S_{2n}.
\]
Thus there is a surjection
\[
(S_{2n}/K^{2n})\otimes_R N
\twoheadrightarrow
(S_{2n}/E_1^nE_2^n)\otimes_R N.
\]
Taking lengths, multiplying by $q!/n^q$, and passing to
the limit yields
\[
\widetilde e_{\rm BR}(E_1E_2;N)
\leq 2^q e_{\rm BR}(K;N).
\]
The expression in terms of mixed multiplicities follows
from Lemma~\ref{lem:mixed-product-formula}.

If $E_1\subseteq E_2$, then $K=E_1$.
Moreover, the surjections
\[
(S_n/E_1^n)\otimes_R N
\twoheadrightarrow
(S_n/E_2^n)\otimes_R N
\]
give $e_{\rm BR}(E_1;N)\geq e_{\rm BR}(E_2;N)$.
Hence the asserted bound by the maximum follows.
The other inclusion is treated symmetrically.

Finally, if $E_1=E_2=E$, then $K=E$ and
\[
\widetilde e_{\rm BR}(E^2;N)
=2^q e_{\rm BR}(E;N),
\]
which proves the equality assertion.
\end{proof}

\section{Integral closure and the converse problem}\label{sec:integral-closure}

Motivated by the Rees--Sharp criterion for ideals recalled in the
Introduction, this section addresses Question~\ref{converse}.
We prove that the implication $\mathrm{(ii)}\Rightarrow\mathrm{(i)}$
holds, and we give counterexamples showing that the converse fails in
general, together with additional positive cases.

\subsection{Equal integral closures and mixed multiplicities}

\begin{prop}\label{prop:equal-closures-mixed-equalities}
Let $E_1,E_2\subsetneq F$ have finite colength and assume that they have the
same integral closure in $F$. Then 
\[
e_{\rm BR}\left(E_1^{[q-i]},E_2^{[i]}\right)
=e_{\rm BR}(E_1)=e_{\rm BR}(E_2)
\qquad(0\leq i\leq q).
\]
Moreover, regarding $E_1E_2$ as a submodule of
$\operatorname{Sym}_R^2(F)$, one has
\[
\widetilde e_{\rm BR}(E_1E_2)
=2^q e_{\rm BR}(E_1).
\]
In particular, the Minkowski-type equality holds:
\[
\widetilde e_{\rm BR}(E_1E_2)^{1/q}
=e_{\rm BR}(E_1)^{1/q}+e_{\rm BR}(E_2)^{1/q}.
\]
\end{prop}

\begin{proof}
Let $H=\overline{E_1}=\overline{E_2}\subseteq F$.
Since $H$ is a submodule of the finite module $F$ over
the Noetherian ring $R$, it is finitely generated.
The inclusion $E_1\subseteq H$ induces a surjection
$F/E_1\twoheadrightarrow F/H$, so $H$ has finite colength.
Furthermore, $E_j\subseteq H=\overline{E_j}$ for $j=1,2$;
and the reduction criterion in Section~\ref{sec:preliminaries} then shows that both $E_1$ and $E_2$ are reductions of $H$.

By invariance under reductions
(Lemma~\ref{lem:mixed-br-basic-properties}\textup{(iii)}) and the formula for
repeated modules (Lemma~\ref{lem:mixed-br-basic-properties}\textup{(ii)}), we obtain
\[
\begin{aligned}
e_{\rm BR}\left(E_1^{[q-i]},E_2^{[i]}\right)
&=e_{\rm BR}\left(H^{[q-i]},H^{[i]}\right)\\
&=e_{\rm BR}(H)
\qquad(0\leq i\leq q).
\end{aligned}
\]
Taking $i=0$ and $i=q$ gives
\[
e_{\rm BR}(E_1)=e_{\rm BR}(H)=e_{\rm BR}(E_2),
\]
which proves the first assertion.

Now Lemma~\ref{lem:mixed-product-formula} and the binomial identity give
\[
\begin{aligned}
\widetilde e_{\rm BR}(E_1E_2)
&=\sum_{i=0}^{q}\binom qi
e_{\rm BR}\left(E_1^{[q-i]},E_2^{[i]}\right)\\
&=e_{\rm BR}(E_1)\sum_{i=0}^{q}\binom qi\\
&=2^q e_{\rm BR}(E_1).
\end{aligned}
\]
Taking $q$-th roots and using
$e_{\rm BR}(E_1)=e_{\rm BR}(E_2)$ yields
\[
\widetilde e_{\rm BR}(E_1E_2)^{1/q}
=2e_{\rm BR}(E_1)^{1/q}
=e_{\rm BR}(E_1)^{1/q}+e_{\rm BR}(E_2)^{1/q}.
\]
\end{proof}

\subsection{A counterexample to the converse}

The following example shows that the converse implication in Question~\ref{converse} fails for modules over a one-dimensional ring.

\begin{ex}\label{ex:mixed-br-equality-distinct-closures}
{\rm Equality of all mixed Buchsbaum--Rim multiplicities does
not imply equality of integral closures, even over a
complete, reduced, Cohen--Macaulay local ring.

Let $k$ be an infinite field and set
\(
R=k[[x,y]]/(xy),\,\, \mathfrak m=(x,y),
\)
where $x,y$ also denote their residue classes in $R$.
The ring $R$ is complete, reduced, Cohen--Macaulay, and
equidimensional of dimension one; in particular, it is
quasi-unmixed. Consider
\[
I=(x,y^2),\qquad J=(x^2,y),\qquad
F=Rf_1\oplus Rf_2,
\]
and set $E_1=IF$ and $E_2=JF$.
Both modules have finite colength in $F$, and
$q=\dim R+\operatorname{rank}F-1=2$.

For integers $a,b\geq1$, the relation $xy=0$ gives
\(
I^aJ^b=(x^{a+2b},y^{2a+b}).
\)
Since
\[
\ell_R\bigl(R/(x^u,y^v)\bigr)=u+v-1
\qquad(u,v\geq1),
\]
we obtain $\ell_R(R/I^aJ^b)=3(a+b)-1$.
Writing $S=\operatorname{Sym}_R(F)=R[t_1,t_2]$, we have
\[
E_1^aE_2^b=I^aJ^bS_{a+b}.
\]
The module $S_{a+b}$ is free of rank $a+b+1$, so
\[
\begin{aligned}
\ell_R(S_{a+b}/E_1^aE_2^b)
&=(a+b+1)\ell_R(R/I^aJ^b)\\
&=(a+b+1)\bigl(3(a+b)-1\bigr)\\
&=3a^2+6ab+3b^2+2a+2b-1.
\end{aligned}
\]
Comparing its quadratic terms with
\[
\frac{e_{\rm BR}(E_1)}{2}a^2
+e_{\rm BR}(E_1^{[1]},E_2^{[1]})ab
+\frac{e_{\rm BR}(E_2)}{2}b^2
\]
yields
\[
e_{\rm BR}(E_1)
=e_{\rm BR}(E_1^{[1]},E_2^{[1]})
=e_{\rm BR}(E_2)=6.
\]

Nevertheless, $\overline{E_1}\neq\overline{E_2}$ in $F$.
Indeed, $xf_1\in E_1$. Suppose that
$xf_1\in\overline{E_2}$. Identifying $f_1,f_2$ with
$t_1,t_2$, respectively, there would exist an equation
\[
(xt_1)^n+a_1(xt_1)^{n-1}+\cdots+a_n=0,
\qquad a_j\in E_2^j,
\]
for some $n\geq1$.
Apply the homomorphism
\[
\varphi:R[t_1,t_2]\longrightarrow k[[x]],
\qquad
x\longmapsto x,\quad y\longmapsto0,\quad
t_1\longmapsto1,\quad t_2\longmapsto0.
\]
Since $\varphi(E_2)\subseteq(x^2)$, this gives
\[
x^n+\varphi(a_1)x^{n-1}+\cdots+\varphi(a_n)=0,
\qquad \varphi(a_j)\in(x^{2j}).
\]
Thus $x$ would be integral over $(x^2)$ in the discrete
valuation ring $k[[x]]$. This is impossible because
$(x^2)$ is integrally closed and $x\notin(x^2)$.
Therefore
\(
xf_1\in\overline{E_1}\setminus\overline{E_2}.
\)
This disproves the implication
\textup{(i)}$\Rightarrow$\textup{(ii)}
in Question~\ref{converse}.}
\end{ex}

\subsection{A positive converse (\ref{converse}) for modules induced by ideals}

We now prove the converse under the additional assumption that the integral closures of the modules are induced by ideals.

\begin{prop}\label{prop:converse-scalar-integral-closures}
Assume in addition that $R$ is quasi-unmixed and $d\geq2$.
Let $E_1,E_2\subsetneq F$ be submodules of finite colength.
Assume that there exist $\mathfrak m$-primary ideals $I,J$
such that
\(
\overline{E_1}=\overline{IF}
\,\,\text{and}\,\,
\overline{E_2}=\overline{JF},
\)
where integral closures are taken in $F$.
Then the following conditions are equivalent:
\begin{enumerate}
\item[\textup{(i)}]
For every $0\leq i\leq q$,
\(
e_{\rm BR}\left(E_1^{[q-i]},E_2^{[i]}\right)
=e_{\rm BR}(E_1)=e_{\rm BR}(E_2).
\)
\item[\textup{(ii)}]
$\overline{E_1}=\overline{E_2}$.
\end{enumerate}
In particular, this equivalence holds when $E_1=IF$
and $E_2=JF$.
\end{prop}

\begin{proof}
The implication \textup{(ii)} $\Rightarrow$ \textup{(i)}
follows from Proposition~\ref{prop:equal-closures-mixed-equalities}.

Conversely, assume \textup{(i)}, and write
$A=e_{\rm BR}(E_1)=e_{\rm BR}(E_2)$.
By invariance under integral closure,
\[
e_{\rm BR}\left((IF)^{[q-i]},(JF)^{[i]}\right)=A
\qquad(0\leq i\leq q).
\]
Let $S=\operatorname{Sym}_R(F)$. Since
\(
(IF)^a(JF)^b=I^aJ^bS_{a+b},
\)
we have
\[
\ell_R\left(S_{a+b}/(IF)^a(JF)^b\right)
=
\binom{a+b+p-1}{p-1}\ell_R(R/I^aJ^b).
\]
Comparing the homogeneous components of total degree $q$
in the corresponding Hilbert polynomials gives
\[
\frac{A}{q!}(X+Y)^q
=
\frac{(X+Y)^{p-1}}{(p-1)!\,d!}
\sum_{j=0}^{d}\binom dj
e\left(I^{[d-j]},J^{[j]}\right)X^{d-j}Y^j.
\]
Canceling $(X+Y)^{p-1}$ and comparing coefficients, we obtain
\[
e\left(I^{[d-j]},J^{[j]}\right)
=\frac{A}{\binom qd}
=e(I)=e(J)
\qquad(0\leq j\leq d).
\]
In particular,
\[
e(IJ)=2^d e(I),
\]
and hence
\[
e(IJ)^{1/d}=e(I)^{1/d}+e(J)^{1/d}.
\]
Since $R$ is quasi-unmixed and $d\geq2$, the equality
criterion in \cite[Theorem~1.16]{Verma} yields positive
integers $r,s$ such that
$\overline{I^r}=\overline{J^s}$.
Taking multiplicities gives
\(
r^d e(I)=s^d e(J).
\)
Thus $r=s$, and consequently $\overline I=\overline J$.

To pass to modules, set $H=I+J$. The equality
$\overline I=\overline J$ implies that both $I$ and $J$
are reductions of $H$. If $H^{n+1}=IH^n$ for all
sufficiently large $n$, then
\[
(HF)^{n+1}
=H^{n+1}S_{n+1}
=IH^nS_{n+1}
=(IF)(HF)^n.
\]
Therefore, $IF$ is a reduction of $HF$; likewise, $JF$
is a reduction of $HF$. It follows that
\[
\overline{E_1}
=\overline{IF}
=\overline{HF}
=\overline{JF}
=\overline{E_2},
\]
as required.
\end{proof}

\begin{rem}
The hypothesis on the integral closures includes modules
that are not themselves of the form $IF$. For example,
the proposition applies whenever $E_1$ and $E_2$ are
reductions of $IF$ and $JF$, respectively.
No assumption on the cardinality of the residue field
is needed.
\end{rem}

\subsection{The converse problem in dimension at least two}

At present, the example over
$k[[x,y]]/(xy)$ only disproves the implication
\textup{(i)}$\Rightarrow$\textup{(ii)} when $d=1$.
The example above does not settle Question~\ref{converse}, which asks
whether the implication holds or fails in dimensions at least two.

\begin{question}
Assume in addition that $R$ is quasi-unmixed, $d\geq2$, and the
residue field is infinite. Let $E_1,E_2\subseteq F$ have finite
colength. If
\[
e_{\rm BR}(E_1^{[q-i]},E_2^{[i]})=c
\qquad\text{for every }0\leq i\leq q,
\]
must one have
\[
\overline{E_1}=\overline{E_2}?
\]
    
\end{question}

\subsection*{Acknowledgments.}
The author thanks Jugal K. Verma for valuable correspondence and
helpful discussions during the preparation of this work.  \subsection*{AI Declaration.}
This project was initiated several years ago, and the vast majority of its theoretical results were obtained without the aid of generative AI tools. More recently, AI technologies were utilized to assist with exposition, literature searches, statement simplification, and proof reviews. Specific concrete calculations were performed with the support of GPT-5.6 Sol and Codex, while Gemini was employed for drafting and polishing the manuscript. The author has thoroughly reviewed all AI-generated suggestions and assumes full responsibility for the final text and mathematical content.

\end{document}